\documentclass[12pt]{amsart}
\usepackage{amsmath}
\usepackage{amssymb}
\usepackage{mathrsfs}
\usepackage[top=3cm,bottom=3.5cm,right=2.5cm,left=2.5cm]{geometry}
\usepackage{ulem}
\newcommand{\Z}{\mathbb{Z}}
\newcommand{\Q}{\mathbb{Q}}
\newcommand{\N}{\mathbb{N}}
\newcommand{\R}{\mathbb{R}}

\newcommand{\SL}{{\rm SL}}

\newcommand{\diag}{{\rm diag}}
\usepackage{amsthm}

\newtheorem{prop}{Proposition}[section]
\newtheorem{conj}[prop]{Conjecture}
\newtheorem{thm}[prop]{Theorem}
\newtheorem{cor}[prop]{Corollary}
\newtheorem{lem}[prop]{Lemma}
\theoremstyle{definition}
\newtheorem{rem}{Remark}[section]
\usepackage{graphicx}%forarXiv

\makeatletter
\@addtoreset{equation}{section}

\makeatother

\title[the number of integers in Littlewood's conjecture]{An improvement of the lower bound of the number of integers in Littlewood's conjecture}
\author{Shunsuke Usuki}
\date{}

\begin{document}
	
	\maketitle
	
	\begin{abstract}
		In this paper, we improve the results in the author's previous paper \cite{Usu22},
		which deals with the quantitative problem on Littlewood's conjecture. We show that, for any $0<\gamma<1$, any $(\alpha,\beta)\in\R^2$ except on a set with Hausdorff dimension about $\sqrt{\gamma}$, any small $0<\varepsilon<1$ and any large $N\in\N$, the number of integers $n\in[1,N]$ such that
		$n\langle n\alpha\rangle\langle n\beta\rangle<\varepsilon$ is greater than $\gamma(\log N)^2/(\log\log N)^2$ up to a universal constant.
	\end{abstract}
	
	\section{Introduction and the main results}
	
	{\it Littlewood's conjecture} is the following famous and long-standing problem in simultaneous Diophantine approximation.
	
	\begin{conj}[Littlewood (c. 1930)]\label{LittlewoodConjecture}
		For every $(\alpha,\beta)\in\R^2$,
		\begin{equation}\label{Littlewood_liminf}
			\liminf_{n\to\infty}n\langle n\alpha\rangle\langle n\beta\rangle=0,
		\end{equation}
		where $\langle x\rangle=\min_{n\in\Z}|x-n|$ for $x\in\R$.
	\end{conj}
	
	In 2000's, Einsiedler, Katok and Lindenstrauss showed in \cite{EKL06} that the set of exceptions to Conjecture \ref{LittlewoodConjecture} is at most zero Hausdorff dimension. Furthermore, the author showed in \cite{Usu22} a quantitative version of this result.
	
	\begin{thm}[{\cite[Theorem 1.3]{Usu22}}]\label{old_maintheorem}
		For any $0<\gamma<1/2$, there exist a subset $Z(\gamma)\subset\R^2$ and constants $C>0$ and $0<\varepsilon_0<1$ independent of $\gamma$ such that
		$\dim_HZ(\gamma)\leq 15\sqrt{\gamma}$ and, for any $(\alpha,\beta)\in \R^2\setminus Z(\gamma)$ and any $0<\varepsilon<\varepsilon_0$,
		$$
		\liminf_{N\to\infty}\frac{1}{\log N}\left|\left\{n\in\{1,\dots,N\}\left|\ n\langle n\alpha\rangle\langle n\beta\rangle<\varepsilon\right.\right\}\right|\geq C\gamma\varepsilon.
		$$
	\end{thm}
	
	By taking $Z=\bigcap_{0<\gamma<1}Z(\gamma)$, we immediately obtain the following corollary.
	
	\begin{cor}[{\cite[Corollary 1.4]{Usu22}}]\label{old_maincorollary}
		There exist $Z\subset\R^2$ and $0<\varepsilon_0<1$ such that $\dim_H Z=0$ and, for any $(\alpha,\beta)\in \R^2\setminus Z$ and any $0<\varepsilon<\varepsilon_0$,
		$$
		\liminf_{N\to\infty}\frac{1}{\log N}\left|\left\{n\in\{1,\dots,N\}\left|\ n\langle n\alpha\rangle\langle n\beta\rangle<\varepsilon\right.\right\}\right|\geq C(\alpha,\beta)\varepsilon
		$$
		for some constant $C(\alpha,\beta)>0$ determined by $(\alpha,\beta)$.
	\end{cor}
	
	Other quantitative results on Conjecture \ref{LittlewoodConjecture} are in \cite{PVZZ22} and \cite{BFG22}.
	
	The aim of this paper is to improve Theorem \ref{old_maintheorem} and Corollary \ref{old_maincorollary}. The following improvements are our main results.
	
	\begin{thm}\label{mainthm}
		For any $0<\gamma<1$, there exists a subset $Z(\gamma)\subset\R^2$ with $\dim_H Z(\gamma)\leq 15\sqrt{\gamma}$ such that, for any $(\alpha,\beta)\in\R^2\setminus Z(\gamma)$ and any $0<\varepsilon<4^{-1}e^{-2}$, we have
		\begin{equation*}
			\liminf_{N\to\infty}\frac{(\log\log N)^2}{(\log N)^2}\left|\left\{n\in\{1,\dots,N\}\left|\ n\langle n\alpha\rangle\langle n\beta\rangle<\varepsilon\right.\right\}\right|\geq \frac{\gamma}{72}.
		\end{equation*}
	\end{thm}
	
	\begin{cor}\label{maincor}
		There exists a subset $Z\subset\R^2$ with $\dim_H Z=0$ such that, for any $(\alpha,\beta)\in\R^2\setminus Z$ and any $0<\varepsilon<4^{-1}e^{-2}$, we have
		\begin{equation*}
			\liminf_{N\to\infty}\frac{(\log\log N)^2}{(\log N)^2}\left|\left\{n\in\{1,\dots,N\}\left|\ n\langle n\alpha\rangle\langle n\beta\rangle<\varepsilon\right.\right\}\right|\geq C_{\alpha,\beta}
		\end{equation*}
		for some positive constant $C_{\alpha,\beta}$ which depends only on $(\alpha,\beta)$, not on $\varepsilon$.
	\end{cor}
	
	\begin{rem}\label{rem_exceptions_unchanged}
		The exceptional set $Z(\gamma)$ in Theorem \ref{mainthm} is the same as that in Theorem \ref{old_maintheorem}.
	\end{rem}
	
	\begin{rem}\label{rem_extension}
		Actually, Theorem \ref{mainthm} holds, without changing the exceptional set $Z(\gamma)$, even if we replace $(\log\log N)^2$ with $18^{-1}f_\gamma(2^{-1}\log N)$ for each $\gamma$, where $f_\gamma:\R_{>0}\to\R_{>0}$ is an arbitrary increasing function such that
		\begin{equation*}
			\frac{\gamma T^2}{f_\gamma(T)}\leq\exp\left(\frac{\sqrt{2}}{3}f_\gamma(T)^{1/2}\right)
		\end{equation*}
		for sufficiently large $T>0$.
	\end{rem}
	
	\begin{rem}
		It seems to be believed that, for Lebesgue a.e. $(\alpha,\beta)\in\R^2$, the order of
		$\left|\left\{n\in\{1,\dots,N\}\left|\ n\langle n\alpha\rangle\langle n\beta\rangle<\varepsilon\right.\right\}\right|$ is $(\log N)^2$. It is shown in \cite{BFG22} that, for $\eta\in(1,2)$, $\delta>0$ and Lebesgue a.e. $(\alpha,\beta)\in\R^2$, we have
		\begin{equation*}
			\left|\left\{n\in\{1,\dots,N\}\left|\ n\langle n\alpha\rangle\langle n\beta\rangle<\varepsilon\right.\right\}\right|=2\varepsilon(\log N)^2+O_{\alpha,\beta,\delta}\left(\varepsilon^{2/\eta}\cdot(\log N)^2\cdot (\log\log N)^{6+\delta}\right)
		\end{equation*}
		for any $0<\varepsilon<1$ and sufficiently large $N\in\N$. This result gives an upper bound. But we can not obtain a lower bound, because the second term can not be negligible when $N\to\infty$. It seems that no further results are known.
	\end{rem}
	
	We can prove Theorem \ref{mainthm} by the same way as the proof in \cite{Usu22} of the previous result Theorem \ref{old_maintheorem}. We only need to replace Theorem 2.1 and Theorem 2.2 in \cite{Usu22} with Theorem \ref{large_entropy} and Theorem \ref{escape_of_mass} below, respectively. Hence, in this paper, we only give proofs for Theorem \ref{large_entropy} and Theorem \ref{escape_of_mass}.
	
	Theorem \ref{large_entropy} and Theorem \ref{escape_of_mass}, the improvements of Theorem 2.1 and Theorem 2.2 in \cite{Usu22}, give an estimate of $\left|\left\{ n\in\N \left|\ n<e^{2T}, n\langle n\alpha\rangle\langle n\beta\rangle<\varepsilon\right.\right\}\right|$ for large $T>1$ when the empirical measures of the element in $\SL(3,\R)/\SL(3,\Z)$ corresponding to $(\alpha,\beta)$ with respect to the full diagonal action {\it converge to a measure with large entropy} and {\it exhibit escape of mass} , respectively. In \cite{Usu22}, it is also shown that Hausdorff dimension of the set of $(\alpha,\beta)$ in the exceptional case is small (\cite[Theorem 1.8]{Usu22}), which leads Theorem 1.2, together with \cite[Theorem 2.1, Theorem 2.2]{Usu22}.
	
	We recall the notations in \cite{Usu22}. Let
	\begin{equation*}
		G=\SL(3,\R),\quad\Gamma=\SL(3,\Z),\quad X=G/\Gamma
	\end{equation*}
	and $m_X$ be the $G$-invariant Borel probability measure on $X$.
	We consider the left action on $X$ by the group of diagonal matrices in $G$ with positive entries
	\begin{equation*}
		A=\left\{\left. \diag(e^{t_1},e^{t_2},e^{t_3})=
		\begin{pmatrix}
			e^{t_1}& & \\ &e^{t_2}& \\ & &e^{t_3}
		\end{pmatrix}
		\right|\ t_1,t_2,t_3\in\R, t_1+t_2+t_3=0
		\right\}.
	\end{equation*}
	We write $a_{s,t}=\diag(e^{-s-t},e^s,e^t)\in A$ for $s,t\in\R$ and $A^+=\{a_{s,t}\left|s,t\geq 0\right.\}\subset A$. For $x\in X$ and $T>0$, we call the Borel probability measure
	\begin{equation*}
		\delta^T_{A^+,x}=\frac{1}{T^2}\int_{[0,T]^2}\delta_{a_{s,t}x}\ dsdt
	\end{equation*}
	on X the {\it $T$-empirical measure of $x$ with respect to the $A^+$-action}. Finally, for $(\alpha,\beta)\in\R^2$, we write
	\begin{equation*}
		\tau_{\alpha,\beta}=
		\begin{pmatrix}
			1& & \\ \alpha&1& \\ \beta& &1
		\end{pmatrix}\in G.
	\end{equation*}
	For a sequence of Borel probability measures $(\mu_k)_{k=1}^{\infty}$ and a finite Borel measure $\mu$ on $X$, $(\mu_k)_{k=1}^\infty$ converges to $\mu$ with respect to the weak* topology if $\int_X f\ d\mu_k\to\int_X f\ d\mu$ as $k\to\infty$ for any $f\in C_0(X)$. We notice that, since $X$ is not compact, $\mu$ might not be a probability measure, but $0\leq \mu(X)\leq 1$ in general.
	Here, we state Theorem \ref{large_entropy} and Theorem \ref{escape_of_mass}.
	
	\begin{thm}\label{large_entropy}
		Let $0<\gamma, p<1$, $\alpha,\beta\in\R$ and $(T_k)_{k=1}^\infty$ be a sequence in $\R_{>0}$ such that $T_k\to\infty$ as $k\to\infty$. Assume that, for $\tau_{\alpha,\beta}\Gamma\in X$, the sequence of the empirical measures $(\delta^{T_k}_{A^+,\tau_{\alpha,\beta}\Gamma})_{k=1}^\infty$ converges to an $A$-invariant Borel probability measure $\mu$ with respect to the weak* topology such that
		\begin{equation*}
			p<\mu(X)\leq 1\quad {\text and}\quad h_{\widehat{\mu}}(a_{1,0})>\gamma
		\end{equation*}
		for $\widehat{\mu}=\mu(X)^{-1}\mu$: the normalization of $\mu$. Then, for every $0<\varepsilon<4^{-1}e^{-2}$, we have
		\begin{equation*}
			\liminf_{k\to\infty}\frac{1}{T_k^2}\left|\left\{ n\in\N \left|\ n<e^{2T_k}, n\langle n\alpha\rangle\langle n\beta\rangle<\varepsilon\right.\right\}\right|\geq p\gamma\varepsilon.
		\end{equation*}
	\end{thm}
	
	\begin{thm}\label{escape_of_mass}
		Let $0<\gamma<1$, $\alpha,\beta\in\R$ and $(T_k)_{k=1}^\infty$ be a sequence in $\R_{>0}$ such that $T_k\to\infty$ as $k\to\infty$. Assume that, for $\tau_{\alpha,\beta}\Gamma\in X$, the sequence of the empirical measures $(\delta^{T_k}_{A^+,\tau_{\alpha,\beta}\Gamma})_{k=1}^\infty$ exhibits $\gamma$-escape of mass, that is, it satisfies
		\begin{equation*}
			\limsup_{k\to\infty}\delta^{T_k}_{A^+,\tau_{\alpha,\beta}\Gamma}(K)\leq 1-\gamma
		\end{equation*}
		for any compact subset $K\subset X$. Then, for any $0<\varepsilon<1/2$, we have
		\begin{equation*}
			\liminf_{k\to\infty}\frac{(\log T_k)^2}{T_k^2}\left|\left\{ n\in\N \left|\ n<e^{2T_k}, n\langle n\alpha\rangle\langle n\beta\rangle<\varepsilon\right.\right\}\right|\geq\frac{\gamma}{18}.
		\end{equation*}
	\end{thm}
	
	\begin{rem} We take $p=1-\gamma$ in Theorem \ref{large_entropy} when we apply this theorem to the proof of Theorem \ref{mainthm}.
	\end{rem}
	
	\begin{rem}\label{extension_theorem_escape_of_mass} Theorem \ref{escape_of_mass} holds even if we replace $(\log T_k)^2$ with $18^{-1}f_\gamma(T_k)$, where $f_\gamma: \R_{>0}\to\R_{>0}$ is an arbitrary increasing function with the property in Remark \ref{rem_extension}. This allows us to extend Theorem \ref{mainthm} as we saw in Remark \ref{rem_extension}.
	\end{rem}
	
	\begin{rem}
		We notice that the orders of the conclusions of Theorem \ref{large_entropy} and \ref{escape_of_mass} are different. The order $(\log N)^2/(\log\log N)^2$ of Theorem \ref{mainthm} is derived from the smaller order $T^2/(\log T)^2$ for $T=\log N$.
	\end{rem}
	
	\section{Proof of Theorem \ref{large_entropy}}\label{proof_large_entropy_case}
	
	In this section, we prove Theorem \ref{large_entropy}. Here, as in the proof of Theorem 2.1 in \cite{Usu22}, the {\it measure rigidity for the full diagonal action on $X=\SL(3,\R)/\SL(3,\Z)$ under positive entropy condition} is essential. This is a result of Einsiedler, Katok and Lindenstrauss (\cite{EKL06}). The following is a modified version of this result, which is easily deduced from the original one and we give a proof in Appendix \ref{proof_measure_rigidity} for the completeness.
	
	\begin{prop}\label{measure_rigidity}
		Let $\mu$ be a Borel probability measure on $X=\SL(3,\R)/\SL(3,\Z)$ which is invariant and ergodic under the action of the discrete subgroup $\{a_{m,n}\}_{m,n\in\Z}$ of $A$. Assume that there exists $(m,n)\in\Z^2\setminus\{0\}$ such that $h_\mu(a_{m,n})>0$. Then, we have $\mu=m_X$.
	\end{prop}
	
	We introduce a classical result by Siegel needed for our proof. For a Borel measurable and bounded function $\varphi: \R^3\to\R$ such that $\{v\in\R^3\left|\ \varphi(v)\neq 0\right.\}$ is bounded, we define the {\it Siegel transform} $\widehat{\varphi}: X\to\R$ by
	\begin{equation*}
		\widehat{\varphi}(x)=\sum_{v\in x\setminus\{0\}}\varphi(v)
	\end{equation*}
	for each $x\in X$. Here, we identify each element $x=g\cdot\SL(3,\Z)\in X=\SL(3,\R)/\SL(3,\Z)$ ($g\in\SL(3,\R)$) with the unimodular lattice $g\cdot\Z^3$ in $\R^3$. Then, we have the following.\footnote{This is needed for a proof of the full statement of Theorem \ref{large_entropy}. However, this is not needed just to prove Theorem \ref{mainthm}, because Theorem \ref{mainthm} says about the order of $(\log N)^2/(\log\log N)^2$ and it is sufficient to show that the $\liminf$ in Theorem \ref{large_entropy} is positive.}
	
	\begin{prop}[Siegel integral formula, {\cite{Sie45}}]\label{Siegel_integral_formula}
		Let $\varphi: \R^3\to\R$ be a Borel measurable and bounded function such that $\{v\in\R^3\left|\ \varphi(v)\neq 0\right.\}$ is bounded. Then, $\widehat{\varphi}\in L^1(m_X)$ and
		\begin{equation*}
			\int_X\widehat{\varphi}\ dm_X=\int_{\R^3}\varphi(v)\ dv.
		\end{equation*}
	\end{prop}
	
	Now, we begin our proof of Theorem \ref{large_entropy}. We first define the open subset $\Omega_{T,\varepsilon}\subset\R^3$ and consider its “tessellation”. The idea of such the tessellation is from \cite{BFG22}, but we use the different one.
	Since Theorem \ref{large_entropy} is trivial if $\alpha\in\Q$ or $\beta\in\Q$, we assume that $\alpha\in\R\setminus\Q$ and $\beta\in\R\setminus\Q$. For $T>1$ and $0<\varepsilon<1$, if we write
	\begin{equation*}
		\Omega_{T,\varepsilon}=\left\{(y,x_1,x_2)\in \R^3\left|\ 0<y<e^{2T},\ 0<|x_1|, |x_2|<\frac{1}{2}, \ y|x_1||x_2|<\varepsilon\right.\right\},
	\end{equation*}
	then it is immediately seen that
	\begin{equation}\label{number_siegel_trans}
		\left|\left\{ n\in\N \left|\ n<e^{2T}, n\langle n\alpha\rangle\langle n\beta\rangle<\varepsilon\right.\right\}\right|
		=\left|\Omega_{T,\varepsilon}\cap (\tau_{\alpha,\beta}\cdot\Z^3)\right|=\widehat{\chi_{\Omega_{T,\varepsilon}}}(\tau_{\alpha,\beta}\Gamma),
	\end{equation}
	where the right-hand side is the Siegel transform of the indicator function of $\Omega_{T,\varepsilon}$. (We notice that $0\notin\Omega_{T,\varepsilon}$.) Let
	\begin{equation*}
		\Delta_\varepsilon=\left\{(y,x_1,x_2)\in \R^3\left|\ 0<y<1,\ \frac{1}{2e}<|x_1|, |x_2|<\frac{1}{2}, \ y|x_1||x_2|<\varepsilon \right.\right\}.
	\end{equation*}
	Then, it can be seen that
	\begin{align*}
		&a_{m,n}^{-1}\Delta_\varepsilon\\
		=&\left\{(y,x_1,x_2)\in \R^3\left|\ 0<y<e^{m+n},\ \frac{e^{-m-1}}{2}<|x_1|<\frac{e^{-m}}{2},\ \frac{e^{-n-1}}{2}<|x_2|<\frac{e^{-n}}{2}, \ y|x_1||x_2|<\varepsilon \right.\right\}
	\end{align*}
	for any $(m,n)\in\Z_{\geq0}^2$, and hence
	\begin{equation*}
		\Omega_{T,\varepsilon}\supset\bigsqcup_{m,n\in\Z_{\geq0}, 0\leq m,n\leq T}a_{m,n}^{-1}\Delta_\varepsilon.
	\end{equation*}
	From this and (\ref{number_siegel_trans}), we obtain that
	\begin{align}\label{tessellation}
		\left|\left\{ n\in\N \left|\ n<e^{2T}, n\langle n\alpha\rangle\langle n\beta\rangle<\varepsilon\right.\right\}\right|
		=&\ \widehat{\chi_{\Omega_{T,\varepsilon}}}(\tau_{\alpha,\beta}\Gamma)\nonumber\\
		\geq&\ \widehat{\chi_{\bigsqcup a_{m,n}^{-1}\Delta_\varepsilon}}(\tau_{\alpha,\beta}\Gamma)\nonumber\\
		=&\sum_{m,n=0}^{\lfloor T\rfloor}\widehat{\chi_{a_{m,n}^{-1}\Delta_\varepsilon}}(\tau_{\alpha,\beta}\Gamma)\nonumber\\
		=&\sum_{m,n=0}^{\lfloor T\rfloor}\widehat{\chi_{\Delta_\varepsilon}}(a_{m,n}\tau_{\alpha,\beta}\Gamma).
	\end{align}
	
	We assume that $(\delta^{T_k}_{A^+,\tau_{\alpha,\beta}\Gamma})_{k=1}^\infty$ converges to an $A$-invariant finite Borel measure $\mu$ such that $p<\mu(X)\leq 1$ and $h_{\widehat{\mu}}(a_{1,0})>\gamma$. To prove Theorem \ref{large_entropy}, it is sufficient to show that, for any given subsequence $(k_l)_{l=1}^\infty$ of $\N$, there exists a subsequence $(k'_j)_{j=1}^\infty$ of $(k_l)_{l=1}^\infty$ such that
	\begin{equation}\label{aim_of_argument_1}
		\liminf_{j\to\infty}\frac{1}{T_{k'_j}^2}\left|\left\{ n\in\N \left|\ n<e^{2T_{k'_j}}, n\langle n\alpha\rangle\langle n\beta\rangle<\varepsilon\right.\right\}\right|\geq p\gamma\varepsilon.
	\end{equation}
	For a given subsequence $(k_l)_{l=1}^\infty$ of $\N$, we consider the sequence of empirical measures $(\delta^{N_{k_l}}_{\tau_{\alpha,\beta}\Gamma})_{l=1}^\infty$ of $\tau_{\alpha,\beta}\Gamma\in X$ with respect to the action of $\{a_{m,n}\}_{m,n\in\Z}$. That is, for each $l\in\N$, we define $N_{k_l}=\lfloor T_{k_l}\rfloor+1$ and a Borel probability measure $\delta^{N_{k_l}}_{\tau_{\alpha,\beta}\Gamma}$ on $X$ by
	\begin{equation*}
		\delta^{N_{k_l}}_{\tau_{\alpha,\beta}\Gamma}=\frac{1}{N_{k_l}^2}\sum_{m,n=0}^{N_{k_l}-1}\delta_{a_{m,n}\tau_{\alpha,\beta}\Gamma}.
	\end{equation*}
	Then, there exists a subsequence $(k'_j)_{j=1}^\infty$ of $(k_l)_{l=1}^\infty$ and a finite Borel measure $\mu'$ on $X$ such that $0\leq\mu'(X)\leq 1$ and $\delta^{N_{k'_j}}_{\tau_{\alpha,\beta}\Gamma}\to\mu'$ as $j\to\infty$ with respect to the weak* topology, that is
	\begin{equation}\label{weak*_conv}
		\int_X\varphi\ d\delta^{N_{k'_j}}_{\tau_{\alpha,\beta}\Gamma}=\frac{1}{N_{k'_j}^2}\sum_{m,n=0}^{N_{k'_j}-1}\varphi(a_{m,n}\tau_{\alpha,\beta}\Gamma)\to\int_X\varphi\ d\mu',\quad j\to\infty
	\end{equation}
	for any $\varphi\in C_0(X)$. From our assumption, we deduce the following lemma on $\mu'$.
	
	\begin{lem}\label{property_mu_dash}
		The finite Borel measure $\mu'$ is invariant under the action of $\{a_{m,n}\}_{m,n\in\Z}$. Furthermore, we have
		\begin{equation*}
			p<\mu'(X)\leq 1\quad {\text and}\quad h_{\widehat{\mu'}}(a_{1,0})>\gamma,
		\end{equation*}
		where $\widehat{\mu'}=\mu'(X)^{-1}\mu'$: the normalization of $\mu'$.
	\end{lem}
	
	\begin{proof}
		Since $N_{k'_j}=\lfloor T_{k'_j}\rfloor+1$ and $T_{k'_j}\to\infty$ as $j\to\infty$, it can be easily seen that $\mu'$ is invariant under the action of $\{a_{m,n}\}_{m,n\in\Z}$.
		
		We take an arbitrary $\psi\in C_0(X)$ and apply (\ref{weak*_conv}) for $\varphi=\int_{[0,1]^2}\psi\circ a_{s,t}\ dsdt\in C_0(X)$. Then, we have
		\begin{align}\label{weak*_conv_integral}
			\int_X\varphi\ d\delta^{N_{k'_j}}_{\tau_{\alpha,\beta}\Gamma}=&\ \frac{1}{N_{k'_j}^2}\sum_{m,n=0}^{N_{k'_j}-1}\int_{[m,m+1]}\int_{[n,n+1]}\psi(a_{s,t}\tau_{\alpha,\beta}\Gamma)\ dsdt\nonumber\\
			=&\ \frac{1}{N_{k'_j}^2}\int_{[0,N_{k'_j}]^2}\psi(a_{s,t}\tau_{\alpha,\beta}\Gamma)\ dsdt\nonumber\\
			\to&\int_X\varphi\ d\mu'\\
			=&\int_X\int_{[0,1]^2}\psi(a_{s,t}x)\ dsdtd\mu'(x)=\int_{[0,1]^2}\int_X\psi(a_{s,t}x)\ d\mu'(x)dsdt,\quad j\to\infty.\nonumber
		\end{align}
		On the other hand, since $N_{k'_j}=\lfloor T_{k'_j}\rfloor+1$ and $T_{k'_j}\to\infty$ as $j\to\infty$, it follows from our assumption on $(T_k)_{k=1}^\infty$ that
		\begin{equation*}
			\lim_{j\to\infty}\frac{1}{N_{k'_j}^2}\int_{[0,N_{k'_j}]^2}\psi(a_{s,t}\tau_{\alpha,\beta}\Gamma)\ dsdt=\lim_{j\to\infty}\int_X\psi\ d\delta^{T_{k'_j}}_{A^+,\tau_{\alpha,\beta}\Gamma}=\int_X\psi\ d\mu.
		\end{equation*}
		From this and (\ref{weak*_conv_integral}), we obtain that $\int_{[0,1]^2}\int_X\psi\circ a_{s,t}\ d\mu'dsdt=\int_X\psi\ d\mu$ for any $\psi\in C_0(X)$, that is,
		\begin{equation*}
			\int_{[0,1]^2}{a_{s,t}}_*\mu'\ dsdt=\mu.
		\end{equation*}
		From this and $p<\mu(X)\leq 1$, we have
		\begin{equation*}
			\mu'(X)=\int_{[0,1]^2}{a_{s,t}}_*\mu'(X)\ dsdt=\mu(X)\in(p,1]
		\end{equation*}
		and, in particular, $\int_{[0,1]^2}{a_{s,t}}_*\widehat{\mu'}\ dsdt=\widehat{\mu}$. By applying \cite[Proposition 2.3]{Usu22} to this integral formula, we obtain
		\begin{equation*}
			\int_{[0,1]^2}h_{{a_{s,t}}_*\widehat{\mu'}}(a_{1,0})=h_{\widehat{\mu}}(a_{1,0}).
		\end{equation*}
		Here, we have, by the property of measure-theoretic entropy, $h_{{a_{s,t}}_*\widehat{\mu'}}(a_{1,0})=h_{\widehat{\mu'}}(a_{1,0})$ and, by the assumption, $h_{\widehat{\mu'}}(a_{1,0})>\gamma$. Hence, from the above equality, we have
		\begin{equation*}
			h_{\widehat{\mu'}}(a_{1,0})=h_{\widehat{\mu}}(a_{1,0})>\gamma,
		\end{equation*}
		and complete the proof of the lemma.
	\end{proof}
	We take the ergodic decomposition of $\widehat{\mu'}$ with respect to the action of $\{a_{m,n}\}_{m,n\in\Z}$ and apply \cite[Proposition 2.3]{Usu22} for measure-theoretic entropy of the action of $a_{1,0}$ to this decomposition. By Proposition \ref{measure_rigidity}, the Haar measure $m_X$ on $X$ is the only Borel probability measure on $X$ which is invariant and ergodic under the action of $\{a_{m,n}\}_{m,n\in\Z}$ such that $h_{m_X}(a_{1,0})>0$ and, it is known that $h_{m_X}(a_{1,0})=4$. (For example, see \cite[Proposition 9.6]{MT94} and \cite[Lemma 6.2]{EK03}.) From this fact and the properties of $\mu'$ stated in Lemma \ref{property_mu_dash}, it follows that, for any Borel measurable non-negative function $\varphi: X\to\R_{\geq0}$,
	\begin{equation}\label{estimate_mu_dash}
		\int_X\varphi\ d\mu'\geq\frac{p\gamma}{4}\int_X\varphi\ dm_X.
	\end{equation}
	
	Let us see (\ref{tessellation}) for $T=T_{k'_j}$. We have
	\begin{align*}
		\frac{1}{T_{k'_j}^2}\left|\left\{ n\in\N \left|\ n<e^{2T_{k'_j}}, n\langle n\alpha\rangle\langle n\beta\rangle<\varepsilon\right.\right\}\right|
		\geq&\ \frac{1}{T_{k'_j}^2}\sum_{m,n=0}^{N_{k'_j}-1}\widehat{\chi_{\Delta_\varepsilon}}(a_{m,n}\tau_{\alpha,\beta}\Gamma)\\
		\geq&\ \frac{1}{N_{k'_j}^2}\sum_{m,n=0}^{N_{k'_j}-1}\widehat{\chi_{\Delta_\varepsilon}}(a_{m,n}\tau_{\alpha,\beta}\Gamma)\\
		=&\int_X\widehat{\chi_{\Delta_\varepsilon}}\ d\delta^{N_{k'_j}}_{\tau_{\alpha,\beta}\Gamma},
	\end{align*}
	and hence,
	\begin{equation}\label{estimate_from_below}
		\liminf_{j\to\infty}\frac{1}{T_{k'_j}^2}\left|\left\{ n\in\N \left|\ n<e^{2T_{k'_j}}, n\langle n\alpha\rangle\langle n\beta\rangle<\varepsilon\right.\right\}\right|\geq\liminf_{j\to\infty}\int_X\widehat{\chi_{\Delta_\varepsilon}}\ d\delta^{N_{k'_j}}_{\tau_{\alpha,\beta}\Gamma}.
	\end{equation}
	We estimate the left-hand side from below. Since $\Delta_{\varepsilon}\subset\R^3$ is open and bounded, it can be seen that the Siegel transform $\widehat{\chi_{\Delta_\varepsilon}}: X\to\R_{\geq0}$ of the indicator function of $\Delta_\varepsilon$ is lower semicontinuous. Therefore, we can take an increasing sequence $(\varphi_l)_{l=1}^\infty$ of non-negative Borel measurable functions such that $\varphi_l(x)\nearrow\widehat{\chi_{\Delta_\varepsilon}}(x)$ as $l\to\infty$ for each $x\in X$, and each $\varphi_l$ is a finite sum of multiples of the indicator function of an open subset on $X$ and a positive constant. For example, we can take
	\begin{align*}
		\varphi_l(x)=&\sum_{i=1}^{l2^l}\frac{1}{2^l}\chi_{\left\{x\in X\left|\ \widehat{\chi_{\Delta_\varepsilon}}(x)>i/2^l\right.\right\}}(x)\\
		=&\begin{cases}
			0, \quad&\text{if }\ 0\leq\widehat{\chi_{\Delta_\varepsilon}}(x)\leq 1/2^l,\\
			i/2^l, \quad&\text{if }\ i/2^l<\widehat{\chi_{\Delta_\varepsilon}}(x)\leq (i+1)/2^l\ \text{ for }\  i=1,\dots,l2^l-1,\\
			l, \quad&\text{if }\ l<\widehat{\chi_{\Delta_\varepsilon}}(x).
		\end{cases}
	\end{align*}
	For each $l$, we have
	\begin{equation*}
		\liminf_{j\to\infty}\int_X\widehat{\chi_{\Delta_\varepsilon}}\ d\delta^{N_{k'_j}}_{\tau_{\alpha,\beta}\Gamma}\geq\liminf_{j\to\infty}\int_X\varphi_l\ d\delta^{N_{k'_j}}_{\tau_{\alpha,\beta}\Gamma}
	\end{equation*}
	and, since $\varphi_l$ is a finite sum of multiples of the indicator function of an open subset on $X$ and a positive constant, it follows from (\ref{weak*_conv}) and (\ref{estimate_mu_dash}) that
	\begin{equation*}
		\liminf_{j\to\infty}\int_X\varphi_l\ d\delta^{N_{k'_j}}_{\tau_{\alpha,\beta}\Gamma}\geq\int_X\varphi_l\ d\mu'\geq\frac{p\gamma}{4}\int_X\varphi_l\ dm_X.
	\end{equation*}
	Hence,
	\begin{equation*}
		\liminf_{j\to\infty}\int_X\widehat{\chi_{\Delta_\varepsilon}}\ d\delta^{N_{k'_j}}_{\tau_{\alpha,\beta}\Gamma}\geq\frac{p\gamma}{4}\int_X\varphi_l\ dm_X
	\end{equation*}
	holds for any $l$. If we take $l\to\infty$, then, by the monotone convergence theorem, we have
	\begin{equation}\label{lower_bound_Haar}
		\liminf_{j\to\infty}\int_X\widehat{\chi_{\Delta_\varepsilon}}\ d\delta^{N_{k'_j}}_{\tau_{\alpha,\beta}\Gamma}\geq\frac{p\gamma}{4}\int_X\widehat{\chi_{\Delta_\varepsilon}}\ dm_X.
	\end{equation}
	By Siegel integral formula (Proposition \ref{Siegel_integral_formula}), we have $\int_X\widehat{\chi_{\Delta_\varepsilon}}\ dm_X=m_{\R^3}(\Delta_\varepsilon)$ ($m_{\R^3}$ is the Lebesgue measure on $\R^3$) and, by standard calculation, it can be seen that $m_{\R^3}(\Delta_\varepsilon)=4\varepsilon$ for $0<\varepsilon<4^{-1}e^{-2}$. From this, (\ref{estimate_from_below}) and (\ref{lower_bound_Haar}), we finally obtain (\ref{aim_of_argument_1}) and complete the proof of Theorem \ref{large_entropy}.
	
	\begin{rem}\label{does_not_work_escape_of_mass}
		We notice that the similar argument does not work in the case of escape of mass (Theorem \ref{escape_of_mass}). In such a case, the accumulation measure $\mu'$ of the sequence of the empirical measures $(\delta^{N_{k'_j}}_{\tau_{\alpha,\beta}\Gamma})_{j=1}^\infty$ has much mass on the point at infinity. On the other hand, the Siegel transform $\widehat{\chi_{\Delta_\varepsilon}}$ we are considering here does not seem to have any continuity at the point at infinity, and “the value of $\widehat{\chi_{\Delta_\varepsilon}}$ at the point at infinity” does not seem to make sense.
		Hence, we can not obtain any meaningful estimate like (\ref{lower_bound_Haar}) in this case.
	\end{rem}
	
	\section{Proof of Theorem \ref{escape_of_mass}}
	
	In this section, we prove Theorem \ref{escape_of_mass}. As we noticed in Remark \ref{does_not_work_escape_of_mass}, we need a different method from that in Section \ref{proof_large_entropy_case}. Here, we start with the same setting as in the proof of Theorem 2.2 in \cite{Usu22}. For $0<\varepsilon<1/2$, we write $B^{\R^3}_\varepsilon=\left\{v\in\R^3\left| \|v\|_\infty<\varepsilon\right.\right\}$, where $\|v\|_\infty=\max\{|v_1|,|v_2|,|v_3|\}$ for $v= {^{t}(v_1,v_2,v_3)}\in\R^3$. We define a subset $X_\varepsilon$ of $X$ by
	\begin{equation*}
		X_\varepsilon=\left\{x\in X\left|\ x\cap \overline{B^{\R^3}_\varepsilon}\neq\{0\}\right.\right\}.
	\end{equation*}
	It can be seen that, by the similar argument as that in Section 2.3 of [Usu22], Theorem \ref{escape_of_mass} is deduced from the following theorem which is an improvement of Proposition 2.4 in \cite{Usu22}.
	
	\begin{thm}\label{key_theorem_escape_of_mass}
		For $\alpha,\beta\in\R\setminus\Q$, $0<\gamma<1$, $0<\varepsilon<1/2$ and $T>e$, we assume that
		\begin{equation}\label{assumption_key_theorem}
			\frac{1}{T^2}m_{\R^2}\left(\left\{(s,t)\in[0,T]^2\left|\ a_{s,t}\tau_{\alpha,\beta}\Gamma\in X_\varepsilon\right.\right\}\right)\geq\gamma,
		\end{equation}
		where $m_{\R^2}$ is the Lebesgue measure on $\R^2$. Then, we have
		\begin{equation*}
			\frac{(\log T)^2}{T^2}\left|\left\{n\in\N\left|\ n<e^{2T}, n\langle n\alpha\rangle\langle n\beta\rangle\leq\varepsilon^3\right.\right\}\right|\geq\frac{\gamma}{18}.
		\end{equation*}
	\end{thm}
	
	Hence, we prove Theorem \ref{key_theorem_escape_of_mass}. First, we see that this theorem is shown as a corollary of the following statement.
	
	\begin{thm}\label{key_theorem_escape_of_mass_2}
		For $\alpha,\beta\in\R\setminus\Q$, $0<\gamma<1$, $0<\varepsilon<1/2$ and $T>1$, we assume (\ref{assumption_key_theorem}). Then, for any $L>1$, we have
		\begin{equation*}
			\left|\left\{n\in\N\left|\ n<e^{2T}, n\langle n\alpha\rangle\langle n\beta\rangle\leq\varepsilon^3\right.\right\}\right|\geq \min\left\{\frac{\gamma T^2}{L}, e^{\sqrt{2}L^{1/2}/3}\right\}.
		\end{equation*}
	\end{thm}
	
	\begin{proof}[Proof of Theorem \ref{key_theorem_escape_of_mass}]
		Assume that Theorem \ref{key_theorem_escape_of_mass_2} holds.
		Let $f: \R_{>e}\to\R_{>1}$ be an arbitrary increasing function such that $f(T)\to\infty$ as $T\to\infty$.
		For $\alpha,\beta\in\R\setminus\Q$, $0<\gamma<1$, $0<\varepsilon<1/2$ and $T>e$, we assume (\ref{assumption_key_theorem}). Then, by applying Theorem \ref{key_theorem_escape_of_mass_2} for $L=f(T)$, we have
		\begin{equation}\label{apply_key_theorem}
			\left|\left\{n\in\N\left|\ n<e^{2T}, n\langle n\alpha\rangle\langle n\beta\rangle\leq\varepsilon^3\right.\right\}\right|\geq \min\left\{\frac{\gamma T^2}{f(T)}, \exp\left(\frac{\sqrt{2}}{3}f(T)^{1/2}\right)\right\}.
		\end{equation}
		Here, we put $f(T)=18(\log T)^2$, then we have
		\begin{equation*}
			\exp\left(\frac{\sqrt{2}}{3}f(T)^{1/2}\right)=\exp(2\log T)=T^2>\frac{\gamma T^2}{18(\log T)^2}
		\end{equation*}
		for any $T>e$. Hence, it follows from (\ref{apply_key_theorem}) that
		\begin{equation*}
			\left|\left\{n\in\N\left|\ n<e^{2T}, n\langle n\alpha\rangle\langle n\beta\rangle\leq\varepsilon^3\right.\right\}\right|\geq\frac{\gamma T^2}{18(\log T)^2}
		\end{equation*}
		and we complete the proof.
	\end{proof}
	
	\begin{rem}
		From the proof of Theorem \ref{key_theorem_escape_of_mass} above, one can easily see that the extension of Theorem \ref{escape_of_mass} as stated in Remark \ref{extension_theorem_escape_of_mass} is possible.
	\end{rem}
	
	What is left is to show Theorem \ref{key_theorem_escape_of_mass_2}. In the rest of this section, we give a proof.
	We use the notions in the proof of Proposition 2.4 in \cite{Usu22}. Let $\alpha,\beta\in\R\setminus\Q$, $0<\gamma<1$, $0<\varepsilon<1/2$ and $T>1$ satisfy (\ref{assumption_key_theorem}). We write
	\begin{equation*}
		\Lambda_{\alpha,\beta,T,\varepsilon}=\left\{n\in\N\left|\ n<e^{2T}, n\langle n\alpha\rangle\langle n\beta\rangle\leq\varepsilon^3\right.\right\}.
	\end{equation*}
	For $\boldsymbol{n}\in\Z^3\setminus\{0\}$, we define
	$$
	d_{\varepsilon,\boldsymbol{n}}=\left\{(s,t)\in\R^2\left|\ \|a_{s,t}\tau_{\alpha,\beta}\boldsymbol{n}\|_\infty\leq\varepsilon\right.\right\}.
	$$
	If $d_{\varepsilon,\boldsymbol{n}}\cap\R_{\geq0}^2\neq \emptyset$
	for  $\boldsymbol{n}={}^t(n,m_1,m_2)\in\Z^3\setminus\{0\}$, then $n\neq0$,
	$$
	d_{\varepsilon,\boldsymbol{n}}=\left\{(s,t)\in\R^2\left|\ s\leq\log\frac{\varepsilon}{|n\alpha+m_1|},\ t\leq\log\frac{\varepsilon}{|n\beta+m_2|},\ s+t\geq\log\frac{|n|}{\varepsilon}\right.\right\}
	$$
	and this is the isosceles right triangle with the length of the leg $\log (\varepsilon^3/|n||n\alpha+m_1||n\beta+m_2|)$. We notice that, since $\alpha,\beta\in\R\setminus\Q$, we have $|n\alpha+m_1|,|n\beta+m_2|>0$. (We describe $d_{\varepsilon, \boldsymbol{n}}$ in Figure 1.)
	\begin{figure}[t]
		\centering
		\includegraphics[width=10cm]{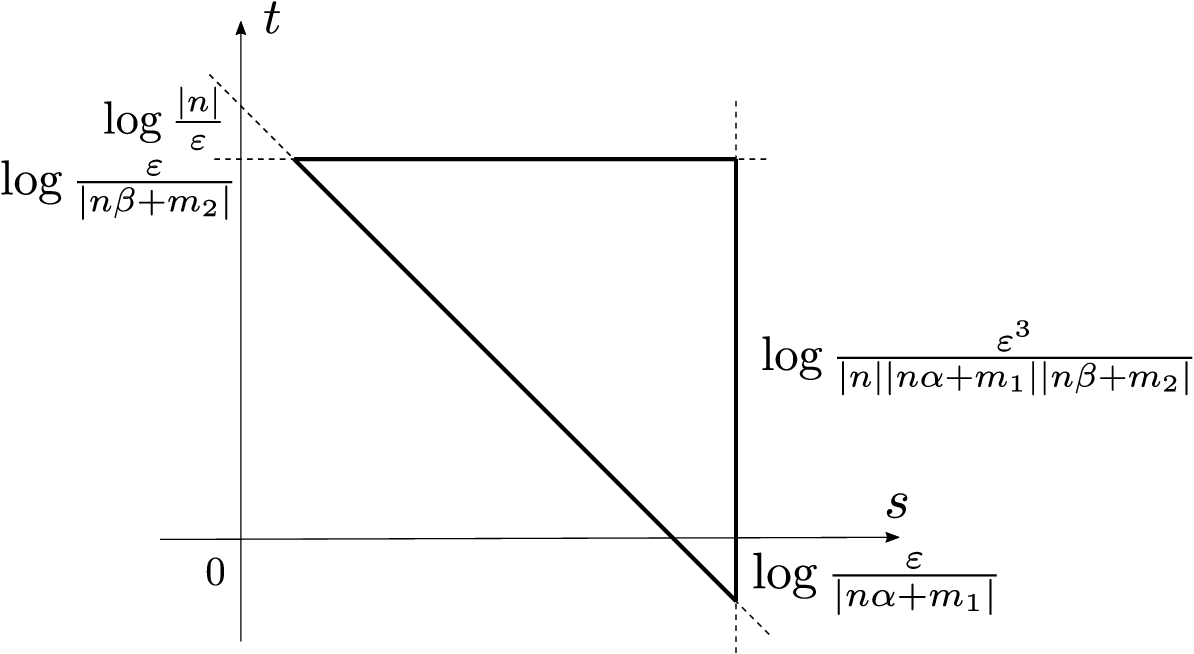}
		\caption{$d_{\varepsilon,\boldsymbol{n}}$}
	\end{figure}
	By the definition of $X_\varepsilon$, it follows that
	\begin{equation}\label{cover_by_d}
		\left\{(s,t)\in[0,T]^2\left|\ a_{s,t}\tau_{\alpha,\beta}\Gamma\in X_{\varepsilon}\right.\right\}=\bigcup_{\boldsymbol{n}\in\Z^3\setminus\{0\}}\left(d_{\varepsilon,\boldsymbol{n}}\cap[0,T]^2\right).
	\end{equation}
	Since $d_{\varepsilon,\boldsymbol{n}}=d_{\varepsilon,-\boldsymbol{n}}$ for each $\boldsymbol{n}\in\Z^3\setminus\{0\}$, we take the union of the right hand side only over $\boldsymbol{n}={}^t(n,m_1,m_1)\in\Z^3\setminus\{0\}$ such that $n>0$. Then, it can be easily seen that,
	for each $n\in\N$, the element $\boldsymbol{n}={}^t(n,m_1,m_2)$ of $\Z^3\setminus\{0\}$ such that
	$d_{\varepsilon, \boldsymbol{n}}\cap[0,T]^2\neq \emptyset$ is at most only one. (See \cite[Lemma 2.5]{Usu22}.)
	For each $n\in\N$, we write $d_{\varepsilon,n}=d_{\varepsilon,\boldsymbol{n}}$ if there exists $\boldsymbol{n}={}^t(n,m_1,m_2)\in\Z^3$ such that $d_{\varepsilon, \boldsymbol{n}}\cap[0,T]^2\neq \emptyset$,
	and $d_{\varepsilon,n}=\emptyset$ if not. Here, it is easily seen from the definition of $d_{\varepsilon, \boldsymbol{n}}$ that, if $d_{\varepsilon, \boldsymbol{n}}\cap[0,T]^2\neq \emptyset$, then we have $n\in\Lambda_{\alpha,\beta,T,\varepsilon}$.
	Therefore, we can write (\ref{cover_by_d}) as
	\begin{equation}\label{cover_by_dn}
		\left\{(s,t)\in[0,T]^2\left|\ a_{s,t}\tau_{\alpha,\beta}\Gamma\in X_{\varepsilon}\right.\right\}=\bigcup_{n\in\Lambda_{\alpha,\beta,T,\varepsilon}}\left(d_{\varepsilon,n}\cap[0,T]^2\right)
	\end{equation}
	and we notice that, for each $n\in\Lambda_{\alpha,\beta,T,\varepsilon}$, the isosceles right triangle $d_{\varepsilon,\boldsymbol{n}}$ with its hypotenuse on the line $\{(s,t)\in\R^2\left|s+t=\log(n/\varepsilon)\right.\}$ is at most only one. (We illustrate the situation in Figure 2.)
	\begin{figure}[h]
		\centering
		\includegraphics[width=12cm]{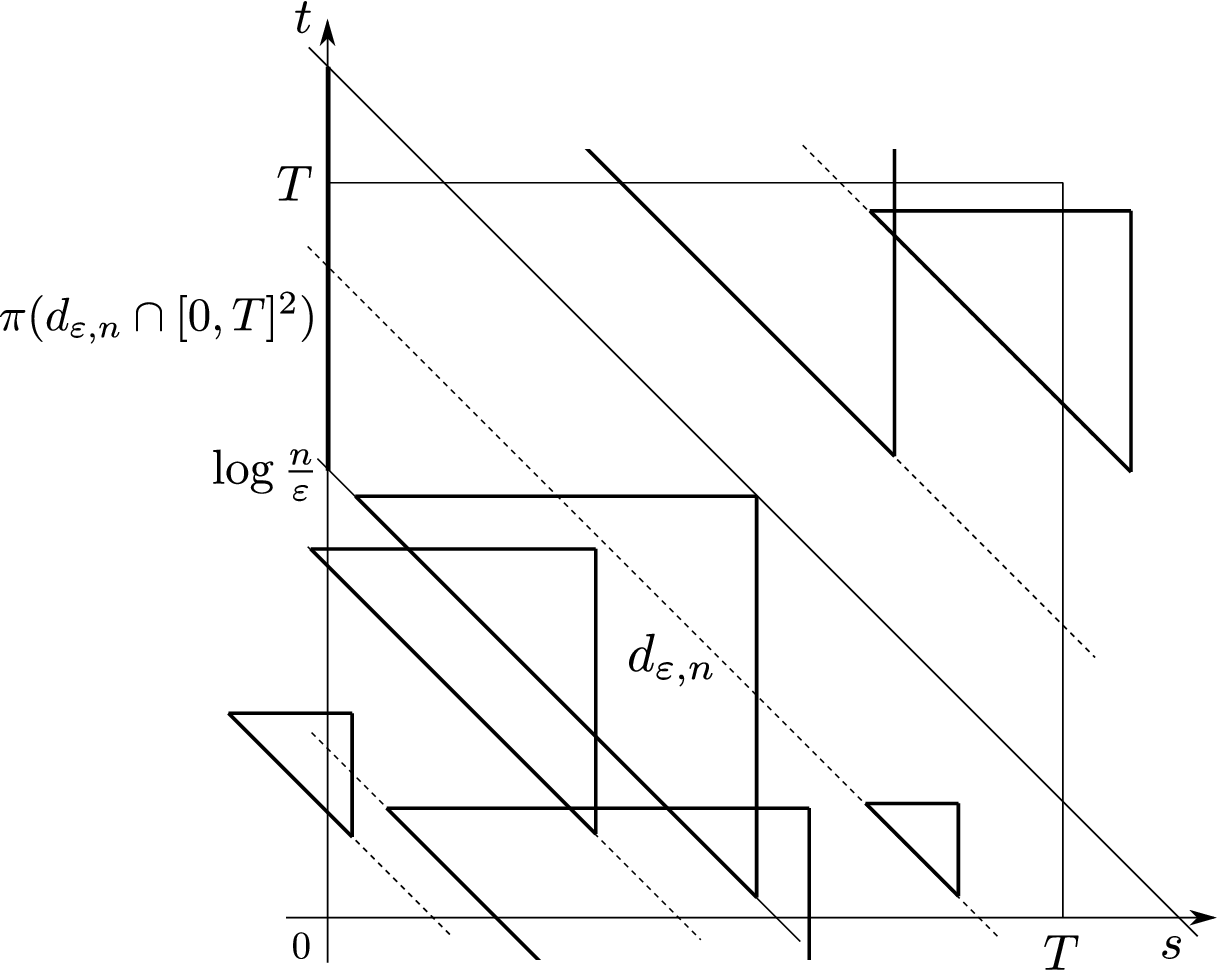}
		\caption{Intersections of $d_{\varepsilon,n} (n\in\Lambda_{\alpha,\beta,T,\varepsilon})$ with $[0,T]^2$ and its image under $\pi$}\label{cover_dn}
	\end{figure}
	
	We take an arbitrary $L>1$.
	To prove Theorem \ref{key_theorem_escape_of_mass_2}, we assume that
	\begin{equation}\label{assumption_proof_key_thm}
		\left|\Lambda_{\alpha,\beta,T,\varepsilon}\right|<\frac{\gamma T^2}{L}.
	\end{equation}
	Since it holds from (\ref{assumption_key_theorem}) and (\ref{cover_by_dn}) that
	\begin{equation*}
		\gamma T^2\leq m_{\R^2}\left(\left\{(s,t)\in[0,T]^2\left|\ a_{s,t}\tau_{\alpha,\beta}\Gamma\in X_\varepsilon\right.\right\}\right)
		\leq\sum_{n\in\Lambda_{\alpha,\beta,T,\varepsilon}}m_{\R^2}\left(d_{\varepsilon,n}\cap[0,T]^2\right),
	\end{equation*}
	under the assumption (\ref{assumption_proof_key_thm}), there exists some $n$ such that
	\begin{equation}\label{large_triangle}
		m_{\R^2}\left(d_{\varepsilon,n}\cap[0,T]^2\right)>L.
	\end{equation}
	We define $\pi: \R^2\to\R$ by $\pi(s,t)=s+t$. We take $(s_{\varepsilon,n},t_{\varepsilon,n})\in d_{\varepsilon,n}\cap[0,T]^2$ such that $s_{\varepsilon,n}+t_{\varepsilon,n}=\max\{s+t\left|\ (s,t)\in d_{\varepsilon,n}\cap[0,T]^2\right.\}$. Since $d_{\varepsilon,n}\cap[0,T]^2$ is an intersection of an isosceles right triangle and a square, we can see that $(s_{\varepsilon,n},t_{\varepsilon,n})\in d_{\varepsilon,n}\cap[0,T]^2$ as above is uniquely determined and
	\begin{equation*}
		\pi\left(d_{\varepsilon,n}\cap[0,T]^2\right)=\left[\log\frac{n}{\varepsilon},s_{\varepsilon,n}+t_{\varepsilon,n}\right].
	\end{equation*}
	(See Figure 2.) Furthermore, $d_{\varepsilon,n}\cap[0,T]^2$ is contained in the isosceles right triangle
	\begin{equation*}
		\widetilde{d_{\varepsilon,n}}=\left\{(s,t)\in\R^2\left|\ s\leq s_{\varepsilon,n}, t\leq t_{\varepsilon,n}, s+t\geq \log\frac{n}{\varepsilon}\right.\right\}.
	\end{equation*}
	Hence, if we write $\lambda_{\varepsilon,n}=s_{\varepsilon,n}+t_{\varepsilon,n}-\log(n/\varepsilon)$ for the length of the interval $\pi\left(d_{\varepsilon,n}\cap[0,T]^2\right)$, we have
	\begin{equation*}
		m_{\R^2}\left(d_{\varepsilon,n}\cap[0,T]^2\right)\leq m_{\R^2}(\widetilde{d_{\varepsilon,n}})=\frac{(\lambda_{\varepsilon,n})^2}{2}.
	\end{equation*}
	From this and (\ref{large_triangle}), we obtain that
	\begin{equation}\label{lambda_large}
		\lambda_{\varepsilon,n}>\sqrt{2}L^{1/2}.
	\end{equation}
	
	We take an arbitrary $k\in\N$ such that $k\leq e^{\sqrt{2}L^{1/2}/3}$. Since $(s_{\varepsilon,n},t_{\varepsilon,n})\in d_{\varepsilon,n}\cap[0,T]^2$, we have
	\begin{equation*}
		s_{\varepsilon,n}+t_{\varepsilon,n}\leq 2T,\quad \langle n\alpha\rangle\leq e^{-s_{\varepsilon,n}}\varepsilon,\quad \langle n\beta\rangle\leq e^{-t_{\varepsilon,n}}\varepsilon.
	\end{equation*}
	From these inequalities, (\ref{lambda_large}) and $\lambda_{\varepsilon,n}=s_{\varepsilon,n}+t_{\varepsilon,n}-\log(n/\varepsilon)$, it follows that
	\begin{equation*}
		kn\leq e^{\sqrt{2}L^{1/2}/3}n<e^{\lambda_{\varepsilon,n}}n=e^{s_{\varepsilon,n}+t_{\varepsilon,n}-\log(n/\varepsilon)}n<e^{2T}
	\end{equation*}
	and
	\begin{align*}
		kn\langle kn\alpha\rangle\langle kn\beta\rangle\leq&\ k^3n\langle n\alpha\rangle\langle n\beta\rangle\\
		\leq&\ e^{\sqrt{2}L^{1/2}}\cdot n\cdot e^{-s_{\varepsilon,n}}\varepsilon\cdot e^{-t_{\varepsilon,n}}\varepsilon\\
		<&\ e^{\lambda_{\varepsilon,n}}\cdot n\cdot e^{-s_{\varepsilon,n}}\varepsilon\cdot e^{-t_{\varepsilon,n}}\varepsilon\\
		=&\ \varepsilon^3.
	\end{align*}
	Therefore, $kn\in\Lambda_{\alpha,\beta,T,\varepsilon}$ and we obtain that
	\begin{equation}\label{conclusion}
		\left\{kn\left|\ k\in\N, k<e^{\sqrt{2}L^{1/2}/3}\right.\right\}\subset\Lambda_{\alpha,\beta,T,\varepsilon}.
	\end{equation}
	
	From the assumption (\ref{assumption_proof_key_thm}) and the conclusion (\ref{conclusion}) of the argument under this assumption, it follows that we have always
	\begin{equation*}
		\left|\Lambda_{\alpha,\beta,T,\varepsilon}\right|\geq \min\left\{\frac{\gamma T^2}{L}, e^{\sqrt{2}L^{1/2}/3}\right\},
	\end{equation*}
	and we complete the proof of Theorem \ref{key_theorem_escape_of_mass_2}.
	
	\appendix
	
	\section{Proof of Proposition \ref{measure_rigidity}}\label{proof_measure_rigidity}
	
	Here, we prove Proposition 2.1 in Section \ref{proof_large_entropy_case}. We first state the original result of measure rigidity for the full diagonal action under positive entropy condition by Einsiedler, Katok and Lindenstrauss, and deduce Proposition 2.1 from it.
	
	\begin{prop}{{\cite[Corollary 1.4]{EKL06}}}\label{measure_rigidity_EKL06}
		Let $\mu$ be an $A$-invariant and ergodic Borel probability measure on $X=\SL(3,\R)/\SL(3,\Z)$ such that $h_\mu(a)>0$ for some $a\in A\setminus\{e\}$. Then, we have $\mu=m_X$.
	\end{prop}
	
	Now, we begin our proof of Proposition \ref{measure_rigidity}. Let $\mu$ be a Borel probability measure on $X=\SL(3,\R)/\SL(3,\Z)$ which is invariant and ergodic under the action of $\{a_{m,n}\}_{m,n\in\Z}$, and assume that $h_\mu(a_{m_0,n_0})>0$ for some $(m_0,n_0)\in\Z^2\setminus\{0\}$. Then, we define a Borel probability measure $\widetilde{\mu}$ on $X$ by
	\begin{equation}\label{definition_mu_tilde}
		\widetilde{\mu}=\int_{[0,1]^2}{a_{s,t}}_*\mu\ dsdt.
	\end{equation}
	Since $\mu$ is invariant and ergodic under the action of $\{a_{m,n}\}_{m,n\in\Z}$, it immediately follows that $\widetilde{\mu}$ is $A$-invariant and ergodic. Furthermore, by the property of measure-theoretic entropy, $h_{{a_{s,t}}_*\mu}(a_{m_0,n_0})=h_{\mu}(a_{m_0,n_0})$ holds for any $(s,t)\in[0,1]^2$, and hence, by applying \cite[Proposition 2.3]{Usu22} to (\ref{definition_mu_tilde}), we have
	\begin{equation*}
		h_{\widetilde{\mu}}(a_{m_0,n_0})=\int_{[0,1]^2}h_{{a_{s,t}}_*\mu}(a_{m_0,n_0})\ dsdt=h_\mu(a_{m_0,n_0})>0.
	\end{equation*}
	Hence, by Proposition \ref{measure_rigidity_EKL06}, $\widetilde{\mu}$ is the Haar measure $m_X$. Therefore, we have
	\begin{equation}\label{Haar_mu_int}
		m_X=\int_{[0,1]^2}{a_{s,t}}_*\mu\ dsdt.
	\end{equation}
	
	It is known that the Haar measure $m_X$ is ergodic under the action of $\{a_{m,n}\}_{m,n\in\Z}$. (For example, this can be shown in the same way as \cite[Corollary 11.19]{EW11}.) Since an ergodic measure can not be represented as a nontrivial convex combination of invariant measures, it follows that ${a_{s,t}}_*\mu=m_X$ for almost every $(s,t)\in\R^2$ with respect to the Lebesgue measure. Finally, by the $A$-invariance of $m_X$, we obtain $m_X=\mu$ and complete the proof.
	
	\subsection*{Acknowledgement} The author is grateful to Masayuki Asaoka and Mitsuhiro Shishikura for their helpful advice. This work was supported by JSPS KAKENHI Grant Number JP23KJ1211.

\end{document}